\documentclass[12pt]{article}

\usepackage[utf8]{inputenc}
\usepackage{graphicx}

\usepackage[all]{xy}
\usepackage{amssymb}
\usepackage{amsmath} 
\usepackage{amsthm}
\usepackage{verbatim}
\usepackage{latexsym}
\usepackage{mathrsfs}

\renewcommand{\labelenumi}{$\mathrm{({\roman{enumi}})}$}
\renewcommand{\emptyset}{\varnothing}

\usepackage{enumitem}

\newcommand{\Sii}{{\Sigma_1^1}}
\newcommand{\Dii}{{\Delta_1^1}}
\newcommand{\Pii}{{\Pi_1^1}}

\renewcommand{\a}{\alpha}
\renewcommand{\b}{\beta}
\newcommand{\g}{\gamma}
\newcommand{\e}{\varepsilon}

\newcommand{\f}{\varphi}
\renewcommand{\k}{\kappa}
\renewcommand{\l}{\lambda}
\renewcommand{\o}{\omega}

\renewcommand{\Cup}{\bigcup}

\newcommand{\WC}{{\operatorname{WC}}}
\newcommand{\DLO}{{\operatorname{DLO}}}
\newcommand{\CLUB}{{\operatorname{CLUB}}}
\newcommand{\NS}{\operatorname{NS}} 
\newcommand{\reg}{\operatorname{reg}}
\newcommand{\Mod}{\operatorname{Mod}}
\newcommand{\dom}{\operatorname{dom}}
\newcommand{\Sk}{\operatorname{Sk}} 
\newcommand{\ZF}{\operatorname{ZF}} 
\newcommand{\cf}{\operatorname{cf}} 
\newcommand{\ran}{\operatorname{ran}}
\newcommand{\Right}{{\operatorname{Right}}}
\newcommand{\Left}{{\operatorname{Left}}}
\newcommand{\Borel}{\operatorname{Borel}} 

\newcommand{\la}{\langle} 
\newcommand{\ra}{\rangle}

\newcommand{\ON}{\mathbb{ON}}

\newcommand{\qo}{\sqsubseteq}
\newcommand{\A}{\mathcal{A}}

\newcommand{\F}{\mathcal{F}}
\renewcommand{\L}{\mathcal{L}}
\newcommand{\G}{\mathcal{G}}

\newcommand{\es}{\varnothing}
\newcommand{\rest}{\restriction}

\renewcommand{\le}{\leqslant}
\renewcommand{\ge}{\geqslant}

\date{}

\newtheorem{theorem}{Theorem}[section]
\newtheorem{corollary}[theorem]{Corollary}
\newtheorem{lemma}[theorem]{Lemma}

\newtheorem{claim}{Claim}[theorem]
\newtheorem*{claimm}{Claim}

\theoremstyle{definition}
\newtheorem{definition}[theorem]{Definition}

\newtheorem{fact}[theorem]{Fact}

\newtheorem*{xrem}{Remark}

\numberwithin{equation}{section}

\newcommand{\proofvpara}{\text{}}

\title{On $\Sigma_1^1$-completeness of quasi-orders on~$\k^\k$}
\author{Tapani Hyttinen\\ University of Helsinki\\ \\Vadim Kulikov\\ University of Helsinki\\ Aalto University\\ \\Miguel Moreno\\ University of Helsinki\\ Bar-Ilan University}

\usepackage[text={15cm,19cm},centering]{geometry}

\begin{document}
\maketitle
\begin{abstract}
  We prove under $V=L$ that the inclusion modulo the non-stationary
  ideal is a $\Sii$-complete quasi-order in the generalized
  Borel-reducibility hierarchy ($\k>\o$). This improvement to known
  results in $L$ has many new consequences concerning the
  $\Sii$-completeness of quasi-orders and equivalence relations such
  as the embeddability of dense linear orders as well as the
  equivalence modulo various versions of the non-stationary
  ideal. This serves as a partial or complete answer to several open
  problems stated in literature. Additionally the theorem is applied
  to prove a dichotomy in $L$: If the isomorphism of a countable
  first-order theory (not necessarily complete) is not $\Dii$, then it
  is $\Sii$-complete.
  
  We also study the case $V\ne L$ and prove $\Sii$-completeness
  results for weakly ineffable and weakly compact~$\k$.
\end{abstract}

\subsection*{Acknowledgments}

The second author wishes to thank the Academy of Finland for the
support through its grant number 285203 as well as both Aalto
University and University of Helsinki for providing suitable research
environment during the academic year 2018-2019.  

The third author wishes to thank the University of Helsinki and its
Doctoral Programme in Mathematics and Statistics (DOMAST) for support.
He also thank Bar-Ilan University and Assaf Rinot for the support
provided by the European Research Council through its grant agreement
ERC-2018-StG~802756.

\section{Introduction}

We work in the setting of generalized descriptive set
theory~\cite{FHK}, GDST for short. The spaces
$\k^\k=\{f\colon \k\to\k\}$ and $2^\k=\{f\colon \k\to 2\}$ are
equipped with the bounded topology where the basic open sets are of
the form $\{\eta\in \k^\k\mid \eta\supset p\}$, $p\in\k^{<\k}$.  Borel
sets are generated by $\k$-long unions and intersection of basic open
sets. Notions of Borel-reducibility between equivalence relations and
quasi-orders as well as Wadge-reducibility between sets are
generalized accordingly. A set is $\Sii$ if it is the projection of a
Borel set, see next section for more detailed definitions.

In \cite{FHK} a Lemma was introduced (a version of the Lemma and a
detailed proof can be found in~\cite[Lemma~1.9 \& Remark~1.10]{HK})
saying that if $V=L$, then any $\Sii$ subset of $\k^\k$ can be
Wadge-reduced to
$$\CLUB=\{\eta\in 2^\k\mid \eta^{-1}\{1\}\text{ contains a }\mu\text{-club}\},\quad\mu<\k\text{ regular,}$$
where ``$\mu$-club'' is short for unbounded set closed under
increasing sequences of length~$\mu$.  In \cite{FHK} this was used to
show that if $V=L$, then $\Sii=\Borel^{*}$.  In \cite{HK} the
Wadge-reducibility result was strengthened by the first two authors of
the present paper.  It was shown (still in~$L$) that every
$\Sii$-equivalence relation is Borel-reducible to the following
equivalence relation on~$\k^\k$:
\begin{equation}
  E^\k_\mu=\{(\eta,\xi)\in (\k^\k)^2\mid \{\a<\k\mid \eta(\a)=\xi(\a)\}\text{ contains a }\mu\text{-club}\}.\label{eq:first}
\end{equation}
We say that $E^\k_\mu$ is $\Sii$-\emph{complete}.

This result was important, but we would have wanted to prove a stronger result, namely that the same equivalence relation on $2^\k$ is $\Sii$-complete:
\begin{equation}
  E^2_\mu=\{(\eta,\xi)\in (2^\k)^2\mid \{\a<\k\mid \eta(\a)=\xi(\a)\}\text{ contains a }\mu\text{-club}\}.\label{eq:second}
\end{equation}
The reason for this was that we
knew many more equivalence relations to which $E^{2}_{\mu}$ can be
Borel-reduced than equivalence relations to which $E^{\k}_{\mu}$ can
be Borel-reduced.  The corollaries of \eqref{eq:first} and
\eqref{eq:second} were explored in~\cite{FHK,HK,MOR}.  In particular,
the question ``Is $E^\k_{\mu}$ Borel-reducible to $E^2_{\mu}$?'' that was stated
in~\cite[Q.~15]{FHK15} and re-stated in~\cite[Q.~3.46]{KLLS} was open
(and it is still open in the general case).  Of course if
$E^{2}_{\mu}$ is $\Sii$-complete the answer to this question is
positive and in the present paper we show that this is the case in~$L$
(Theorem~\ref{thm:maincor}) by first proving a corresponding result for
quasi-orders (Theorem~\ref{thm:QOV=L}). Borel-reducibility between
quasi-orders is a natural generalization of reducibility between
equivalence relations (see Section~\ref{sec:Preliminaries} for precise
definitions).

We then prove a range of new results which are all consequences of
Theorem~\ref{thm:QOV=L}. One of these is our main result: If $V=L$,
then the isomorphism relation of \emph{any} countable first-order
theory (not necessarily complete) is either $\Dii$ or $\Sii$-complete.
A closely related classification problem in the generalized Baire
space was studied in \cite{HKM}, the so-called ``Borel-reducibility
counterpart of the Shelah's main gap theorem''. The other results of
this paper are partial answers to \cite[Q.'s~11.3~and~11.4]{Luc}
(which are re-stated as \cite[Q's~3.49~and~3.50]{KLLS}),
\cite[Q.~15]{FHK15} and a complete answer to~\cite[Q.~3.47]{KLLS}.

These questions ask about the (consistency of) reducibility between
relations of the form $E^{\l}_{\mu}$, $\l\in\{2,\k\}$,
$\mu\in\reg(\k)$, quasi-orders arising as subset relations modulo
certain ideals (like the $\mu$-non-stationary ideal), quasi-orders of
embeddability between linear orders as well as various isomorphism
relations. In particular, \cite[Q.~11.4]{Luc} asks whether the
embeddability of dense linear orders $\qo_{\DLO}$ is a $\Sii$-complete
quasi-order for weakly compact~$\k$. From those results that are
described above it follows that $\qo_{\DLO}$ is $\Sii$-complete in $L$
for all $\k$ that are not successors of an $\o$-cofinal cardinal
(Theorem~\ref{thm:DLOVisLSiiCompl}). Since $\qo_{\DLO}$ is
Borel-reducible to $\qo_G$, the embeddability of graphs, this
quasi-order is also $\Sii$-complete in this scenario.  In
Section~\ref{ssec:Ineffable} we extend this to weakly ineffable
cardinals (without the assumption $V=L$).  Thus the only case in which
\cite[Q.~11.4]{Luc} is still open is the case when $V\ne L$ and $\k$
is a weakly compact cardinal which is not weakly ineffable. In
Section~\ref{ssec:SiiCompIsoDLO} we prove that the isomorphism of DLO,
$\cong_{\DLO}$, on $\k$ weakly compact is $\Sii$-complete (here again,
we do not assume $V=L$) and this implies the same for $\cong_G$, the
isomorphism of graphs. The existence of $\Sii$-complete isomorphism
relations has been previously known to hold in $L$ \cite{HK18}.  It is
still unknown whether there exists a model of ZFC and $\k>\o$ on which
no isomorphism relation on models of size $\k$ is $\Sii$-complete (a
stark contrast to the case $\k=\o$ where the isomorphism relation on
any class of countable structures is induced by a Polish group action
and therefore not
$\Sii$-complete~\cite{kechris1997classification}). Given the present
situation such a counterexample will have to satisfy both $V\ne L$ and
$\k$ is not weakly compact.


\section{Preliminaries and Definitions}
\label{sec:Preliminaries}

In this section we define the notions and concepts we work with.
Throughout this article we assume that $\kappa$ is an uncountable
cardinal that satisfies $\k^{<\k}=\k$ which is a standard assumption
in the GDST. In this paper, however, this assumption is mostly redundant,
because we work either with strongly inaccessible $\k$ or under the
assumption $V=L$. For sets $X$ and $Y$ denote by $X^Y$ the set of all
functions from $Y$ to $X$.  For ordinal $\a$ denote by $X^{<\a}$ the
set of all functions from any $\beta <\a$ to~$X$.  We work with the
generalized Baire and Cantor spaces associated with $\kappa$ these
being $\k^\k$ and $2^\k$ respectively, where $2=\{0,1\}$. The
generalized Baire space $\kappa^\k$ is equipped with the bounded
topology. For every $\zeta\in \kappa^{<\kappa}$, the set
$$\{\eta\in \kappa^\kappa \mid \zeta\subset \eta\}$$ 
is a basic open set. The open sets are of the form $\bigcup X$ where
$X$ is a collection of basic open sets. The collection of
\emph{$\k$-Borel} subsets of $\kappa^\kappa$ is the smallest set which
contains the basic open sets and is closed under unions and
intersections, both of length $\kappa$.  A \emph{$\k$-Borel set} is
any element of this collection.  In this paper we do not consider any
other kind of Borel sets, so we always omit the prefix~\mbox{``$\kappa$-''}.
The subspace $2^\k\subset \k^\k$ (the generalized Cantor space) is
equipped with the subspace topology. We will also work in the
subspaces of the form $\Mod^\k_T$ which are sets of codes for models
with domain $\k$ of a first-order countable theory~$T$. Special cases
include $\Mod^\k_G$ and $\Mod^\k_{\DLO}$ for
graphs and dense linear orders respectively.  These are Borel
subspaces of $2^\k$. This enables us to view the quasi-order of
embeddability of models, say $\qo_{\DLO}$, as a quasi-order on~$2^\k$.
In order to precisely define this, we have to introduce some notions.

The following is a standard way to code structures with domain
$\kappa$ by elements of $\k^\kappa$ (see e.g.~\cite{FHK}). Suppose
$\L=\{P_n\mid n<\omega\}$ is a countable relational
vocabulary.

\begin{definition}\label{def:coding}
  Fix a bijection $\pi\colon \kappa^{<\omega}\to \kappa$. For every
  $\eta\in 2^\kappa$ define the $\L$-structure
  $\A_\eta$ with domain $\kappa$ as follows: For every
  relation $P_m$ with arity $n$, every tuple $(a_1,a_2,\ldots, a_n)$
  in $\kappa^n$ satisfies
  $$(a_1,\ldots, a_n)\in P_m^{\A_\eta}\Longleftrightarrow \eta(\pi(m,a_1,\ldots,a_n))=1.$$
\end{definition}

Note that for every $\L$-structure $\A$ with
$\dom(\A)=\k$ there exists $\eta\in 2^\kappa$ with
$\A= \A_\eta$. It is clear how this coding can be
modified for a finite vocabulary. 
For club many $\alpha<\kappa$ we can
also code the $\L$-structures with domain $\alpha$:

\begin{definition}\label{def:clubstr}
  Denote by $C_\pi$ the club
  $\{\alpha<\kappa\mid \pi[\alpha^{<\omega}]\subseteq \alpha\}$. For
  every $\eta\in 2^\kappa$ and every $\alpha\in C_\pi$ define the
  structure $\A_{\eta\restriction \alpha}$ with domain
  $\alpha$ as follows: For every relation $P_m$ with arity $n$, every
  tuple $(a_1,a_2,\ldots , a_n)$ in $\alpha^n$ satisfies
  $$(a_1,a_2,\ldots , a_n)\in P_m^{\A_{\eta\restriction
      \alpha}}\Longleftrightarrow (\eta\restriction\alpha)
  (\pi(m,a_1,a_2,\ldots,a_n))=1.$$
\end{definition}

Note that for every $\alpha\in C_\pi$ and every $\eta\in 2^\k$ the
structures $\A_{\eta\restriction \a}$ and
$\A_\eta\restriction\a$ are the same.

Let us denote by $\Mod^\k_T$ the subset of $2^\k$ consisting of those
elements that code the models of a first-order countable theory $T$ (not
necessarily complete). Abbreviate first-order countable theory as FOCT
from now on. We will be interested in particular in $T=G$, the theory
of graphs (symmetric and irreflexive) and $T=\DLO$, the theory of
dense linear orders without end-points.  We consider $\Mod^\k_T$ as a
topological space endowed with the subspace topology. 

We can now define some central relations for this paper. A
\emph{quasi-order} is a transitive and reflexive relation.

\begin{definition}[Relations]\label{def:ISO}
  We will use the following relations.
  \begin{description}
  \item[Isomorphism] For a FOCT $T$, define
    $$\cong^\k_T\ =\ \cong_T\ =\{(\eta,\xi)\in 2^\kappa\times 2^\kappa\mid
    \eta,\xi\in \Mod^\k_T, \A_\eta\cong \A_\xi\text{
      or } \eta,\xi\notin \Mod^\k_T\}.$$
  \item[Embeddability] For a FOCT $T$, define the quasi-order
    $$\qo^\k_T\ =\ \qo_{T}\ =\{(\eta,\xi)\in (\Mod^\k_T)^2\mid \A_\eta\text{ is embeddable into }\A_{\xi}\}$$
    Thus, for example $\qo_G$ is the embeddability of graphs and
    $\qo_{\DLO}$ is the embeddability of dense linear orders.
  \item[Bi-embeddability] For a FOCT $T$ and
    $\eta,\xi\in \Mod^\k_{T}$, let
    $$\eta \approx_T \xi\iff \eta \qo_T \xi \land \xi\qo_T \eta.$$
  \item[Inclusion mod NS] For $\eta,\xi\in 2^\k$ and a stationary
    $S\subseteq\k$, we write
    $\eta\qo^{S} \xi$ if $(\eta^{-1}\{1\}\backslash \xi^{-1}\{1\})\cap S$
    is non-stationary.
  \item[Equivalence mod NS] For every stationary $S\subseteq \kappa$
    and $\l\in\{2,\k\}$, we define $E^{\l}_S$ as the relation
    $$E^{\l}_S=\{(\eta,\xi)\in \l^\k\times \l^\k\mid \{\a<\k\mid
    \eta(\a)\ne \xi(\a)\}\cap S \text{ is not stationary}\}.$$
    Note that $\eta E^2_S \xi$ if and only if $\eta\qo^S \xi \land \xi\qo^S \eta.$
  \end{description}
\end{definition}

If $S$ is the set of all $\mu$-cofinal ordinals, denote
$E^\l_{S}=E^\l_{\mu}$ and $\qo^S=\qo^\mu$. If $S$ is the set of all
regular cardinals below $\k$, denote $S=\reg(\k)=\reg$ in which case
$E^\l_S=E^\l_{\reg}$ and $\qo^S=\qo^{\reg}$. If $S=\k$, write
$E^\l_S=E^\l_{\NS}$ and $\qo^S=\qo^{\NS}$

A quasi-order $Q$ on (a Borel set) $X\subseteq \k^\k$ is $\Sigma_1^1$,
if $Q\subseteq X^2$ is the projection of a closed set in
$X^2\times \k^\k$ ($X$ is equipped with subspace topology and
$X^2\times \k^\k$ with the product topology). All quasi-orders of
Definition~\ref{def:ISO} (note that equivalence relations are
quasi-orders) are~$\Sii$.

Suppose $X,Y\subseteq \k^\k$ are Borel. A function
$f\colon X\to Y$ is \emph{Borel}, if for every open set
$A\subseteq Y$ the inverse image $f^{-1}[A]$ is a Borel subset of $X$
with respect to the induced Borel structure on $X$ and $Y$. 

If $Q_1$ and $Q_2$ are quasi-orders respectively on $X$ and $Y$, then
we say that $Q_1$ is \emph{Borel-reducible} to $Q_2$, if there exists a
Borel map $f\colon X\to Y$ such that for all $x_1,x_2\in X$ we have
$x_1 Q_1 x_2\iff f(x_1) Q_2 f(x_2)$ and this is also denoted by
$Q_1\le_B Q_2$. If $f$ is continuous (inverse image of an open set is
open), then we say that $Q_1$ is \emph{continuously reducible} to
$Q_2$.  Note that equivalence relations are quasi-orders, so this
gives naturally a notion of reducibility for them as well.
We will interchangeably use the notations $xEy$ and $(x,y)\in E$
if $E$ is a binary relation, because we consider it as a set of pairs.
Sometimes one notation is clearer than the other.

Note that if we define $F\colon 2^\k\to 2^\k$ by 
$$F(\eta)(\a)=
\begin{cases}
  \eta(\a) &\mbox{if } \a\in S\\
  1 & \mbox{otherwise }
\end{cases}
$$ 
for a fixed $S\subseteq\k$, we obtain:

\begin{fact}\label{reg-red-NS}
  For all stationary $S\subseteq S'$ we have $\qo^S\,\le_B\,\qo^{S'}$. \qed
\end{fact}

A quasi-order is \emph{$\Sigma_1^1$-complete}, if every
$\Sigma_1^1$ quasi-order is Borel-reducible to it.  An
equivalence relation is \emph{$\Sii$-complete} if every
$\Sii$ equivalence relation is Borel-reducible to it.

A Borel equivalence relation $E$ on a Borel subspace $X\subseteq 2^\k$
can be extended to $2^\k$ by declaring all other elements equivalent
to each other, but not equivalent to any of the elements in $X$.
Similarly a quasi-order $\qo$ on $X\subseteq 2^\k$ can be trivially
extended to the whole space $2^\k$. If the original equivalence
relation or quasi-order was $\Sigma^1_1$-complete, then so are the
extensions.


\section{$\Sigma_1^1$--completeness of $\qo^S$ in~$L$}

This section is devoted to proving Theorem \ref{thm:QOV=L}.  In
Section~\ref{sec:Corollaries} a range of corollaries will be proved.

\begin{theorem}\label{thm:QOV=L}
  $(V=L, \k>\o)$ The quasi-order $\qo^\mu$ is $\Sigma_1^1$--complete, for
  every regular $\mu<\k$.
\end{theorem}

As mentioned in Introduction, this is an improvement to a theorem in
\cite{HK} which says that $E^\k_{\mu}$ is $\Sii$-complete.

\begin{definition} We will need a version of the diamond principle.
  Denote by $\ON$ the class of all ordinals.
  \begin{itemize}
  \item Let us define a class function $F_\Diamond\colon \ON\to L$.
    For all $\a$, $F_\Diamond(\a)$ is a pair $(X_\a,C_\a)$ where
    $X_\a,C_\a\subseteq \a$, if $\a$ is a limit ordinal, then $C_\a$
    is either a club or the empty set, and $C_\a=\emptyset$ when $\a$
    is not a limit ordinal. We let $F_\Diamond(\a)=(X_\a,C_\a)$ be the
    $<_L$-least pair such that for all $\beta\in C_\a$,
    $X_\beta\neq X_\a\cap\beta$ if $\a$ is a limit ordinal and such
    pair exists and otherwise we let
    $F_\Diamond(\a)=(\emptyset,\emptyset)$.
  \item We let $C_\Diamond\subseteq \ON$ be the class of all limit
    ordinals $\a$ such that for all $\beta<\a$,
    $F_\Diamond\restriction\beta\in L_\a$. Notice that for every
    regular cardinal $\a$, $C_\Diamond\cap\a$ is a club.
  \end{itemize}
\end{definition}

\begin{definition}\label{diamondseq}
  For a given regular cardinal $\a$ and a subset $A\subseteq \a$, we
  define the sequence $(X_\gamma,C_\gamma)_{\gamma\in A}$ to be
  $(F_\Diamond(\gamma))_{\gamma\in A}$, and the sequence
  $(X_\gamma)_{\gamma\in A}$ to be the sequence of sets $X_\gamma$
  such that $F_\Diamond(\gamma)=(X_\gamma,C_\gamma)$ for some
  $C_\gamma$.
\end{definition}

By $S^\a_\mu$ we denote the set of all $\mu$-cofinal ordinals
below~$\a$.

\begin{xrem}
  It is known that if $\a$ and $\mu$ are regular cardinals such that
  $\mu<\a$, then the sequence $(X_\gamma)_{\gamma\in S_\mu^\a}$ is a
  diamond sequence (i.e. for all $Y\subseteq \a$, the set
  $\{\gamma\in S_\mu^\a\mid Y\cap\gamma=X_\gamma\}$ is
  stationary). Notice that if $\beta\in C_\Diamond$, then for all
  $\gamma<\beta$, $X_\gamma\in L_\beta$.
\end{xrem}

By $\ZF^-$ we mean ZFC$+(V=L)$ without the power set axiom. By
$\ZF^\diamond$ we mean $\ZF^-$ with the following axiom:

\begin{quote}
  ``Let $(S_\gamma,D_\gamma)=F_\Diamond(\gamma)$ for all $\g<\a$ and
  $\mu<\a$ a regular cardinal. Then $(S_\gamma)_{\gamma\in S^\a_\mu}$
  is a diamond sequence.''
\end{quote}

Whether or not $\ZF^{-}$ proves $\ZF^{\diamond}$ is
irrelevant for the present argument.  We denote by $\Sk(Y)^{L_\theta}$
the Skolem closure of $Y$ in $L_\theta$ under the definable Skolem
functions.

\begin{lemma}\label{reflection-in-L}
  $(V=L)$ For any $\Sigma_1$-formula $\varphi(\eta,x)$ with parameter
  $x\in 2^\k$ and a regular cardinal $\mu<\k$, then for all $\eta\in 2^\k$
  we have:
  \begin{enumerate}
  \item If $\varphi(\eta,x)$ holds, then $A$ contains a club,
  \item If $\varphi(\eta,x)$ does not hold, then 
    $S\backslash A$ is $\mu$-stationary,
  \end{enumerate}
  where $S=\{\a\in S^\k_\mu\mid X_\a=\eta^{-1}\{1\}\cap\a\}$ and
  $$A=\big\{\a\in C_\Diamond\cap \k\mid \exists \beta>\a\big(L_\beta\models
  \ZF^\diamond\wedge \varphi(\eta\restriction
  \a,x\restriction\a)\wedge r(\a)\big)\big\}$$
  where $r(\alpha)$ is the formula ``$\alpha$ is a regular
  cardinal''.
\end{lemma}
\begin{xrem}
  This Lemma is reminiscent of \cite[Remark 1.10]{HK}, but there is a
  big difference, because now $S$ depends on $\eta$ through the
  diamond-sequence which makes this Lemma stronger. The proof in
  \cite{HK} is not applicable here.
\end{xrem}
\begin{proof}
  Let $\mu<\k$ be a regular cardinal.  Suppose that $\eta\in 2^\k$ is
  such that $\varphi(\eta,x)$ holds. Let $\theta>\k$ be a cardinal large
  enough such that
  $$L_\theta\models \ZF^\diamond\wedge \varphi(\eta,x)\wedge
  r(\k).$$
  For each $\alpha<\kappa$,
  let $$H(\alpha)=\Sk(\alpha\cup\{\k,\eta,x\})^{L_\theta}$$ and
  $\bar{H}(\alpha)$ the Mostowski collapse of
  $H(\alpha)$. Let $$D=\{\alpha<\k\mid H(\alpha)\cap\k=\alpha\}.$$
  Then $D$ is a club set and $D\cap C_\Diamond$ is a club. Since
  $H(\alpha)$ is an elementary submodel of $L_\theta$ and the
  Mostowski collapse $\bar{H}(\alpha)$ is equal to $L_\beta$ for some
  $\beta>\a$, we have $D\cap C_\Diamond\subseteq A$.
  This proves~i.
  
  Suppose $\eta\in 2^\k$ is such that $\varphi(\eta,x)$ does not
  hold. Let $\mu<\k$ be a regular cardinal. 
  Let $C$ be an arbitrary unbounded set which is closed under $\mu$-limits (a
  $\mu$-club). We will show that $C\cap (S\setminus A)$ is non-empty which
  by the arbitrariness of $C$ implies that $S\setminus A$ is $\mu$-stationary,
  as desired.

  Let $\theta>\k$ be a large enough cardinal such that 
  $$L_\theta\models \ZF^\diamond\wedge \neg\varphi(\eta,x)\wedge r(\k).$$
  Let
  $$H(\alpha)=\Sk(\alpha\cup\{\k,C,\eta,x,(X_\gamma,C_\gamma)_{\gamma\in S^\k_\mu}\})^{L_\theta}.
  $$
  Let
  $$D=\{\alpha\in S_\mu^\k\mid H(\alpha)\cap\k=\alpha\}$$
  Then $D$ is an unbounded set, closed under $\mu$-limits.  Notice
  that since $H(\a)$ is an elementary substructure of $L_\theta$, then
  $H(\a)$ calculates all cofinalities correctly below $\a$. Let
  $S=\{\a\in S^\k_\mu\mid X_\a=\eta^{-1}\{1\}\cap\a\}$ and $\alpha_0$
  be the least ordinal in $(\lim{\!}_\mu D)\cap S$ (where
  $\lim{\!}_\mu D$ is the set of ordinals of $D$ that are
  $\mu$-cofinal limits of elements of $D$). By the elementarity of
  each $H(\alpha)$ we conclude that $\alpha_0\in C$.  It remains to show
  that $\alpha_0\notin A$ to complete the proof.
  
  Let $\bar{\beta}$ be such that $L_{\bar{\beta}}$ is equal to the
  Mostowski collapse of $H(\alpha_0)$.  Suppose, towards a contradiction, that
  $\alpha_0\in A$. Thus $\a_0\in C_\Diamond\cap\k$ and there exists
  $\beta>\alpha_0$ such that
  $$L_\beta\models \ZF^\diamond\wedge \varphi(\eta\restriction
  \alpha_0,x\restriction \alpha_0)\wedge r(\alpha_0).$$
  Since $\varphi(\eta,x)$ is a $\Sigma_1$-formula which holds in
  $L_{\beta}$ and not in $L_{\bar\beta}$, $\beta$ must be greater than
  $\bar{\beta}$. It must be a limit ordinal because $L_\b\models \ZF^{-}$.
  
  \begin{claim}
    $L_\beta$ satisfies the following:
    \begin{enumerate}
    \item For all $\gamma\in S\cap \a_0$, $\gamma$ has cofinality
      $\mu$.
    \item $S\cap \a_0$ is a stationary subset of $\a_0$.
    \item $D\cap\a_0$ is a $\mu$-club subset of $a_0$.
    \end{enumerate}
  \end{claim}
  
  \begin{proof}
    \begin{enumerate}
    \item $H(\a_0)$ calculates all cofinalities correctly below
      $\a_0$. Thus $L_{\bar{\beta}}$ calculates all cofinalities
      correctly below $\a_0$. Since $\beta$ is greater than
      $\bar{\beta}$, $L_{\beta}$ also calculates all cofinalities correctly
      below $\a_0$. Since $S\cap \a_0\subseteq S^\k_\mu$ in $L$, we have
      that $S\cap \a_0\subseteq S^\k_\mu$ holds in $L_{\beta}$.
    \item Since $\a_0\in C_\Diamond\cap\k$ and $L_\beta$ satisfies
      $\ZF^\diamond$ and $r(\a_0)$, $L_\beta$ satisfies that
      $S\cap \a_0$ is a stationary subset of $\a_0$.
    \item Being unbounded in $\a_0$ is absolute between $L$ and
      $L_\beta$ and since $\a_0\in\lim{\!}_\mu D$, $D\cap \a_0$ is
      unbounded in $\a_0$, so it remains to show that $D\cap \a_0$
      is closed under $\mu$-limits in~$L_\b$.

      Let $\a<\a_0$ be such that
      $L_\beta\models \cf(\a)=\mu\ \wedge\ \bigcup(D\cap\a)=\a$, we
      will show that $L_\beta\models \a\in D\cap \a_0$. Since
      $L_{\beta}$ calculates all cofinalities correctly below $\a_0$,
      $L\models \cf(\a)=\mu\wedge\ \bigcup(D\cap\a)=\a$. $D$ is a
      $\mu$-club in $L$, thus $L\models\a\in D$. Since $\a<\a_0$,
      $L\models\a\in D\cap\a_0$. We will finish the proof by showing
      that $L\models\a\in D\cap\a_0$ implies
      $L_\beta\models\a\in D\cap\a_0$.
      
      Notice that $H(\alpha_0)$ is a definable subset of $L_\theta$
      and $D$ is a definable subset of $L_\theta$. By elementarity,
      $D\cap\alpha_0$ is a definable subset of $H(\alpha_0)$, we
      conclude that $D\cap\alpha_0$ is a definable subset of
      $L_{\bar{\beta}}$ and $D\cap\alpha_0\in L_\beta$. Therefore
      $L_\beta\models\a\in D\cap\a_0$.
    \end{enumerate}
  \end{proof}
  
  Since $L_\beta\models r(\a_0)$, by the previous claim we concluded
  that $L_\beta$ satisfies ``$\lim{\!}_\mu D\cap \a_0$ is a
  $\mu$-club''. Since $S\cap \a_0$ is a stationary subset of $\a_0$ in
  $L_\beta$, we conclude that
  $$L_\beta\models (\lim{\!}_\mu D\cap \a_0)\cap S\cap\a_0\neq\emptyset,$$ so $$L\models (\lim{\!}_\mu D
  \cap \a_0)\cap S\cap\a_0\neq\emptyset.$$
  This contradicts the minimality of $\a_0$.
\end{proof}

Now we are ready to prove Theorem~\ref{thm:QOV=L}.

\begin{proof}[Proof of Theorem~\ref{thm:QOV=L}]
  Suppose $Q$ is a $\Sigma_1^1$ quasi-order on $\k^\k$. Let
  $a\colon\k^\k\to 2^{\k\times\k}$ be the map defined
  by $$a(\eta)(\alpha,\beta)=1\Leftrightarrow \eta(\alpha)=\beta.$$
  Let $b$ be a continuous bijection from $2^{\k\times\k}$ to $2^\k$,
  and $c= b\circ a$. Define $Q'\subset 2^\k\times 2^\k$ by
  $$(\eta,\xi)\in Q'\Leftrightarrow (\eta=\xi)\vee (\eta,\xi\in
  \ran(c)\wedge (c^{-1}(\eta),c^{-1}(\xi))\in Q)$$
  So $c$ is a continuous reduction of $Q$ to $Q'$, and $Q'$ is a
  $\Sigma_1^1$ quasi-order because it is a continuous image of $Q$.
  On the other hand $Q'$ is a quasi-order on $2^\k$ and not on $\k^\k$
  like the original $Q$ was. Hence, we can assume, without loss of
  generality, that $Q$ is a quasi-order on~$2^\k$.
  
  There is a $\Sigma_1$-formula of set theory
  $\psi(\eta,\xi)=\psi(\eta,\xi,x)=\exists k\varphi(k,\eta,\xi,x)\vee
  \eta=\xi$ with $x\in 2^\k$, such that for all $\eta,\xi\in 2^\k$,
  $$(\eta,\xi)\in Q\Leftrightarrow \psi(\eta,\xi),$$ we added
  $\eta=\xi$ to $\psi(\eta,\xi)$, to ensure that when we reflect
  $\psi(\eta\restriction \a,\xi\restriction\a)$ we get a reflexive
  relation.  Let $r(\alpha)$ be the formula ``$\alpha$ is a regular
  cardinal'' and $\psi^Q(\k)$ be the sentence with parameter $\k$ that
  asserts that $\psi(\eta,\xi)$ defines a quasi-order on $2^\k$.  For
  all $\eta\in 2^\k$ and $\alpha<\kappa$, let
  $$T_{\eta,\alpha}=\{p\in 2^\alpha\mid \exists
  \beta>\alpha(L_\beta\models \ZF^\diamond\wedge
  \psi(p,\eta\restriction \alpha,x\restriction\a)\wedge
  r(\alpha)\wedge \psi^Q(\a))\}.$$
  Let $(X_\a)_{\a\in S_\mu^\k}$ be the diamond sequence of
  Definition~\ref{diamondseq}, and for all $\a\in S_\mu^\k$, let
  $\chi_\a$ be the characteristic function of $X_\a$. Define
  $\F\colon 2^\k\to 2^\k$ by
  $$\F(\eta)(\alpha)=
  \begin{cases}
    1 &\mbox{if } \chi_\a\in T_{\eta,\a} \text{ and } \a\in S_\mu^\k\\
    0 & \mbox{otherwise }
  \end{cases}$$
  
  \begin{claimm}
    If $\eta\ Q\ \xi$, then $T_{\eta,\a}\subseteq T_{\xi,\a}$ for
    club-many $\a$'s.
  \end{claimm}
  
  \begin{proof}
    Suppose $\psi(\eta,\xi,x)=\exists k\varphi(k,\eta,\xi,x)$ holds
    and let $k$ witness that. Let $\theta$ be a cardinal large
    enough such that
    $L_\theta\models \ZF^\diamond\wedge \varphi(k,\eta,\xi,x)\wedge
    r(\kappa)$.
    For all $\alpha<\k$ let
    $H(\alpha)=\Sk(\alpha\cup\{\k,k,\eta,\xi,x\})^{L_\theta}$. The set
    $D=\{\alpha<\k\mid H(\alpha)\cap\k=\alpha\wedge
    H(\alpha)\models\psi^Q(\a)\}$
    is a club. Using the Mostowski collapse we have that
    $$D'=\{\alpha<\k\mid \exists \beta>\alpha(L_\beta\models \ZF^\diamond\wedge \varphi(k
    \restriction \alpha,\eta\restriction \alpha,\xi\restriction
    \alpha,x\restriction \alpha) \wedge r(\alpha)\wedge
    \psi^Q(\a))\}$$
    contains a club. For all $\alpha\in D'$ and $p\in T_{\eta,\alpha}$
    we have that
    $$\exists \beta_1>\alpha(L_{\beta_1}\models \ZF^\diamond\wedge
    \psi(p,\eta\restriction \alpha,x\restriction \a)\wedge
    r(\alpha)\wedge \psi^Q(\a))$$ and
    $$\exists \beta_2>\alpha(L_{\beta_2}\models \ZF^\diamond\wedge \psi(\eta\restriction \alpha,
    \xi\restriction \alpha,x\restriction\a)\wedge r(\alpha)\wedge
    \psi^Q(\a)).$$
    Therefore, for $\beta=\max\{\beta_1,\beta_2\}$ we have that
    $$L_{\beta}\models \ZF^\diamond\wedge \psi(p,\eta\restriction \alpha,x\restriction \a)\wedge
    \psi(\eta\restriction \alpha,\xi\restriction \alpha,x\restriction
    \a)\wedge r(\alpha)\wedge \psi^Q(\a).$$
    Since $\psi^Q(\a)$ holds and so transitivity holds for
    $\psi(\eta,\xi)$ in $L_\beta$, we conclude that
    $$L_{\beta}\models \ZF^\diamond\wedge \psi(p,\xi\restriction
    \alpha,x\restriction\a)\wedge r(\alpha)\wedge \psi^Q(\a)$$
    so $p\in T_{\xi,\alpha}$ and
    $T_{\eta,\alpha}\subseteq T_{\xi,\alpha}$. This holds for all
    $\alpha\in D'$.
  \end{proof}
  
  By the previous claim, we conclude that if $\eta\ Q\ \xi$, then
  there is a $\mu$-club $C$ such that for every $\a\in C$ it holds
  that
  $\chi_\a\in T_{\eta,\alpha}\Rightarrow \chi_\a\in T_{\xi,\alpha}$.
  Therefore
  $(\F(\eta)^{-1}\{1\}\backslash\F(\xi)^{-1}\{1\})\cap
  C=\emptyset$,
  and $\F(\eta)\ \qo^{\mu}\ \F(\xi)$.
  
  For the other direction, suppose $\neg\psi(\eta,\xi,x)$ holds. Let
  $S=\{\a\in S_\mu^\k\mid X_\a=\eta^{-1}\{1\}\cap\a\}$. Since
  $(X_\gamma)_{\gamma\in S_\mu^\k}$ is a diamond sequence, $S$ is a
  stationary set. By Lemma \ref{reflection-in-L} we know that
  $S\backslash A$ is stationary, where
  $$A=\{\a\in C_\Diamond\cap\k \mid \exists \beta>\a(L_\beta\models
  \ZF^\diamond\wedge\psi(\eta\restriction
  \a,\xi\restriction\a,x\restriction\a)\wedge r(\a))\}.$$
  Since for all $\a\in S\backslash A$ we have that
  $X_\a=\eta^{-1}\{1\}\cap\a$, so $\chi_\a\in T_{\eta,\alpha}$.
  We conclude that for all $\a\in S\backslash A$,
  $\F(\eta)(\alpha)=1$. On the other hand, for all
  $\a\in S\backslash A$ it holds that
  $$\forall\beta>\a(L_\beta\not\models
  \ZF^\diamond\wedge\psi(\eta\restriction\a,\xi\restriction\a,x\restriction\a)\wedge
  r(\a))$$ so
  $$\forall\beta>\a(L_\beta\not\models \ZF^\diamond\wedge\psi(\chi_\a,\xi\restriction\a,x
  \restriction\a)\wedge r(\a)).$$ Therefore
  $$\forall\beta>\a(L_\beta\not\models \ZF^\diamond\wedge\psi(\chi_\a,\xi\restriction\a,x
  \restriction\a)\wedge r(\a)\wedge \psi^Q(\a))$$
  we conclude that $\chi_\a\not\in T_{\xi,\alpha}$, and
  $\F(\xi)(\alpha)=0$. Hence, for all $\a\in S\backslash A$,
  $\F(\eta)(\alpha)=1$ and
  $\F(\xi)(\alpha)=0$. Since $S\backslash A$ is stationary,
  we conclude that
  $\F(\eta)^{-1}\{1\}\backslash\F(\xi)^{-1}\{1\}$ is
  stationary and
  $\F(\eta)\ \not\qo^{\mu}\ \F(\xi)$.
\end{proof}


\section{Corollaries to Theorem~\ref{thm:QOV=L}}
\label{sec:Corollaries}

\subsection{$\Sii$-completeness of $E^2_{\mu}$ in $L$}

\begin{theorem}[$V=L$, $\k>\o$]\label{thm:NScomV=L}
  $\qo^{\NS}$ is a $\Sigma_1^1$-complete quasi-order.
\end{theorem}
\begin{proof}
  Follows from Fact \ref{reg-red-NS} and Theorem~\ref{thm:QOV=L}.
\end{proof}

\begin{theorem}[$V=L$]\label{thm:maincor}
  Let $\mu$ be a regular cardinal below $\k$, then
  $E^2_{\mu}$ is a $\Sigma^1_1$-complete equivalence
  relation.
\end{theorem}
\begin{proof}
  This follows from Theorem \ref{thm:QOV=L}, because
  $E^2_{\mu}$ is a symmetrization of the
  quasi-order~$\qo^\mu$.
\end{proof}

The above result cannot be proved in ZFC. It was shown in
\cite[Thm~56]{FHK} that if $\k$ is not a successor of a singular
cardinal, then in a cofinality preserving forcing extension
$E^2_{\mu_1}$ and $E^2_{\mu_2}$ are
$\le_B$-incomparable for regular cardinals $\mu_1<\mu_2<\k$.

Theorem~\ref{thm:NScomV=L} gives consistently a positive answer to
``Given a weakly compact cardinal $\k$, is $\qo^{\NS}$ complete?''
\cite[Q.~11.4]{Luc}.  Theorem~\ref{thm:maincor} answers the questions
``Is it consistently true that $E^2_{\mu}\le_B E^2_{\lambda}$ for
$\lambda<\mu$?''  \cite{FHK},\cite[Q.~3.47]{KLLS} (take $\l=\o$,
$\mu=\o_1$ and $\k=\o_2$), and gives consistently a positive answer to
``Is $E^\k_{\mu}$ Borel-reducible to $E^2_{\mu}$ for a
regular~$\mu$?''  \cite[Q.~15]{FHK15}, \cite[Q.~3.46]{KLLS}.

\subsection{$\Sii$-completeness of $\qo_{\DLO}$ and $\qo_{G}$ in~$L$}

\cite[Q.~11.3]{Luc} asks ``Given a weakly compact cardinal $\k$, is
$\qo_{\DLO}$ complete for $\Sii$ quasi-orders? What about arbitrary
regular cardinals $\k$?'' In this section we apply
Theorem~\ref{thm:QOV=L} to show that the answer is positive if~$V=L$.
To do that we first have to establish a general theorem about 
$\qo_{\DLO}$:

\begin{theorem}\label{thm:ReducionNStoDLO}
  Suppose that for all $\l<\k$ we have $\l^\o<\k$. Then there is a
  continuous reduction of $\qo^\o$ to~$\qo_{\DLO}$.
\end{theorem}
\begin{proof}
  Fix an $\o$-club $G\subseteq S^\k_\o\setminus (\o+1)$ with the property that for all
  $\alpha<\k$ and all $\beta<\k$ there exists $\g<\k$ with
  $\beta<\g<\k$ such that $[\g,\g+\a]\cap G=\es$, where
  $[\g,\g+\a]=\{\delta < \k \mid \g\le \delta\le\g+\a\}$, thus $G$ is
  in a sense ``sparse''. Such a $G$ can be obtained by constructing a
  sequence $(\g_\a)_{\a<\k}$ as follows. Let $\g_0=\o+1$, for successor
  $\b=\a+1$ let $\g_{\b}=\g_\a+\g_\a$ and for limit $\b$ let
  $\g_\b=\sup\{\g_\a\mid \a<\b\}$. Then let $G=\{\g_\a\mid \a<\k\}\cap S^\k_\o$,
  For a subset $A\subseteq \k$, denote
  $$A_G=((A\cap G)\cup (\k\setminus G))\setminus \{\omega\}$$ 
  On the one hand $A_G$ is
  equivalent to $A$ up to the $\o$-non-stationary ideal. On the other
  hand $A_G$ contains arbitrarily long intervals. Note also that $A_G$
  contains all ordinals of uncountable cofinality, $S^\k_{>\o}\subset A_G$.

  For ordinals $\a,\b$, we say that a function $f\colon \a\to\b$ is
  \emph{continuous} if it is continuous with respect to the order
  topology, that is, for all increasing sequences
  $(\g_\delta)_{\delta<\l}\subset\a$, $\l<\a$, we have
  $$\sup\{f(\g_\delta)\mid \delta<\l\} = f(\sup\{\g_\delta\mid \delta<\l\}).$$

  \begin{claim}\label{cl:shift}
    Suppose $A$ and $B$ are subsets of $\k$. Then $A\setminus B$ is
    $\o$-non-stationary if and only if there exists a strictly
    increasing continuous $f\colon \k\to\k$ such that
    $$f[A_G]\subseteq B_G$$
  \end{claim}
  \begin{proof}
    From the definition of $A_G$ and $B_G$ we see that
    $A_G\setminus B_G=(A\setminus B)\cap G\cap S^\k_\o$. Since $G$ is
    an $\o$-club, we have that $A_G\setminus B_G$
    is $\o$-stationary if and only if $A\setminus B$ is. For any
    $f\colon \k\to\k$ which is increasing and continuous the set
    $C_f=\{\a<\k\mid f(\a)=\a\}$ is club. Thus, if $A\setminus B$ is
    $\o$-stationary, then
    $(A_G\cap C_f)\setminus B_G=(A_G\setminus B_G)\cap C_f$ is also
    $\o$-stationary and therefore non-empty. This proves the direction
    ``from right to left'' of the claim.

    Assume now that $A\setminus B$ is not $\o$-stationary and let
    $C_1\subseteq S^\k_\o$ be an $\o$-club such that
    $A\cap C_1\subseteq B$. Let $C=C_1\cap G$. Note that now not only
    $A\cap C\subseteq B$, but also $A_G\cap C\subseteq B_G$.  We will
    define $f\colon \k\to\k$ by inductively building a sequence of
    strictly increasing continuous functions $f_\a\in \k^{<\k}$ such
    that 
    \begin{enumerate}
    \item if $\a<\b$, then $f_\a\subset f_\b$,
    \item the domain of $f_\a$ is a successor ordinal $\e+1$ for some $\e\in C\cup S^\k_{>\o}$,
      where $S^\k_{>\o}$ is the set of ordinals below $\k$ with uncountable cofinality.
    \item if $\a<\b$, then $\ran(f_\a)\subset \dom(f_\b)$,
    \item if $f_\a(\delta)=\delta$, then $\delta\in C\cup S^\k_{>\o}$, 
    \item if $f_\a(\delta)\ne \delta$, then $f_\a(\delta)\in B_G$.
    \end{enumerate}
    Before constructing this sequence let us show that
    $f=\Cup_{\a<\k}f_\a$ will be the desired function. It is strictly
    increasing and continuous, because $f_{\a}$ are.  To show that
    $f[A_G]\subset B_G$, suppose that $\gamma\in f[A_G]$.  Now
    $\gamma=f(\delta)$ for some $\delta\in A_G$ and clearly
    $\gamma= f_\a(\delta)$ for sufficiently large~$\a$.  By (5), if
    $\gamma=f_\a(\delta)\ne \delta$, we have $\gamma\in B_G$ and we
    are done. So we may assume that $\gamma=f_\a(\delta)=\delta$.
    Since $\delta\in A_G$ we now also have $\g\in A_G$.  By (4) we
    have $\g\in C\cup S^\k_{>\o}$. If $\g\in C$, then since
    $C\cap A_G\subseteq B_G$, we have $\g\in B_G$ so we are done.
    If $\g\in S^\k_{>\o}$, then by the definition of $G$
    it must be the case that $\g\in \k\setminus G$. But $\k\setminus G\subset B_G$,
    so again $\g\in B_G$.

    Thus, it remains to construct the sequence $(f_\a)_{\a<\k}$
    satisfying (1)--(5).  Let $f_0=\es$. If $f_{\a}$ is defined, then
    define $f_{\a+1}$ as follows. Let $\e_\a=\max\dom f_\a$. By (2)
    $\e_\a$ is well defined and we have $\dom f_\a=[0,\e_\a]$ and
    $\e_\a\in C\cup S^\k_{>\o}$.  Let $\e_{\a+1}$ be some ordinal such that
    $\e_{\a+1}>f_\a(\e_\a)$ and $\e_{\a+1}\in C$. Then find
    $\g_0>\e_{\a+1}$ such that $[\g_0,\g_0+\e_{\a+1}]\subset B_G$
    which is possible by the definition of $G$ and the fact that
    $\k\setminus G\subset B_G$. Now define for all $\delta\le \e_{\a+1}$
    
    $$f_{\a+1}(\delta)=
    \begin{cases}
      f_\a(\delta)& \text{if }\delta\le\e_\a\\
      \g_0+\delta & \text{otherwise},
    \end{cases}$$
    Clearly (1) and (2) are satisfied for $f_{\a+1}$. Since $\g_0>0$, we have
    $f_{\a+1}(\delta)>\delta$ for all
    $\delta\in [\e_\a+1, \e_{\a+1}]$, so if $f_\a$ satisfies (4), 
    then so does $f_{\a+1}$. Because of the choice of $\g_0$, (5)
    is satisfied. Also (3) is satisfied by the choice of $\e_{\a+1}$.

    If $\b$ is a limit and $f_\a$ is defined for $\a<\b$, then let
    $\e_\b=\sup_{\a<\b}\e_{\a}$.  From (3) and that $f_\a$ are
    increasing it follows that
    $$\e_{\b}=\sup_{\a<\b}\ran f_{\a}=\sup_{\a<\b}\dom f_{\a}.$$
    The domain of $f_\b$ is $\e_\b + 1$ and it is defined by:  
    $$f_{\b}(\delta)=
    \begin{cases}
      f_\a(\delta)& \text{if }\delta<\e_\b,\text{ for some }\a\text{ such that }\delta<\e_\a<\e_\b\\
      \delta & \text{otherwise},
    \end{cases}$$
    Clause (3) ensures that this is well-defined.  Clauses (1) and (3)
    are clearly satisfied for~$f_\b$. If $\e_\b$ has cofinality $\o$,
    then (2) is satisfied, because $\e_\b$ is the limit of elements of
    $C$ and $C$ is $\o$-club. Otherwise, if $\cf(\e_\b)>\o$, then (2)
    is trivially satisfied. The only new element in $\dom f_{\b}$ that
    is not in $\dom f_{\a}$ for any $\a<\b$ is $\e_\b$ and since
    $\e_\b\in C\cup S^\k_{>\o}$, (4) and (5) are satisfied by
    induction.
  \end{proof}

  For every $p,q\in \k^{\le\o}$ define $p\prec q$ if either
  $p\supset q$ or there exists $n<\o$ such that $p(n)\ne q(n)$ and for
  the smallest such $n$ we have $p(n)<q(n)$.  This defines a linear
  order on the set $C(\k^{\le\o})$ of all strictly increasing
  functions $p\in \k^{\le\o}$.
  
  Now for $A\subseteq\k$ define the linear order $L(A)$ to be
  the set
  $$\{p\in C(\k^{\le\o})\mid \dom p=\o\text{ and }
  \sup\ran p \in A\text{ and }p(0)=0\}$$
  equipped with the order~$\prec$.  Note that because elements of
  $L(A)$ have domain $\o$, the condition $p\supset q$ in the
  definition of $p\prec q$ is never relevant for them.  This is a
  modification of a construction given by Baumgartner \cite{Bau}.
  Clearly $A\subseteq B$ implies $L(A)\subseteq L(B)$.  We will show
  that $A\mapsto L(A_G)$ is a reduction of $\qo^\o$
  to~$\qo_{\DLO}$. By the definition of $A_G$, the limit ordinals of $A_G$
  (which are the only ones that matter in the definition of $L(A_G)$) are all
  greater $\o$ which ensures that there is no smallest
  element in~$L(A_G)$. Also clearly $L(A_G)$ does not have a greatest
  element, because $A_G\cap S^\k_\o$ is unbounded and it is dense by the following
  argument. If $p\prec q$ for elements of $L(A_G)$, then (because
  $\dom p=\dom q=\omega$) there is $n<\o$ with $p(n)<q(n)$. Let
  $p'\colon\o\to\k$ be defined by $p'(k)=p(k)$ for $k\le n$ and
  $p'(k)=p(k)+1$ otherwise. Then $p\prec p'\prec q$.
  
  If $f\colon \k\to\k$ is continuous and strictly increasing with
  $f(0)=0$ and $A\subseteq \k$ any set, the definition of $L(A)$
  implies that
  $$\{f\circ p\mid p\in L(A)\}\subseteq L(f[A]).$$ 
  
  Thus, if $f\colon \k\to\k$ is continuous and strictly increasing
  such that $f[A_G]\subseteq B_G$, then $p\mapsto f\circ p$ defines an
  embedding from $L(A_G)$ into $L(B_G)$. By Claim~\ref{cl:shift} such
  $f$ exists, if $A\qo^\o B$ (we do not lose generality by assuming
  $f(0)=0$).
  
  The other direction is essentially a simplification of the proof of
  Baumgartner Theorem 5.3(ii) \cite{Bau}. If $A\not\qo^\o B$, then, as
  noted above, also $A_G\not\qo^\o B_G$ and so $A_G\setminus B_G$ is
  $\o$-stationary. So it is sufficient to show that for
  \emph{any} unbounded $A,B\subseteq S^\k_\o$, if $A\setminus B$ is
  $\o$-stationary, then $L(A)$ cannot be embedded into $L(B)$. 
  
  So suppose that $A\setminus B$ is stationary and assume towards a
  contradiction that $h\colon L(A) \to L(B)$ preserves the
  ordering~$\prec$.  For any $X\subseteq C(\k^{\le\o})$, let
  $T(X)=\{p\in C(\k^{\le\o})\mid \exists q\in X(p\subseteq q)\}$.
  Note that for every strictly increasing $p\in \k^{<\o}$ with
  $p(0)=0$, we have $p\in T(L(A))$ and $p\in T(L(B))$.  For
  $g\in T(L(B))$, let
  $$\Right(g)=\{f\in L(A)\mid h(f)=g\text{ or }g\prec h(f)\},$$
  $$\Left(g)=\{f\in L(A)\mid h(f)\prec g\}.$$
  Let
  $$\rho(g)=\{f'\in T(\Right(g))\mid \text{ for all }g'\in T(\Right(g)), \text{ if }g'\prec f'
  \text{, then }f'\subseteq g'\},$$
  $$\lambda(g)=\{f'\in T(\Left(g))\mid \text{ for all }g'\in T(\Left(g)), \text{ if }f'\prec 
  g'\text{, then }g'\subseteq f'\}.$$
  Note that $\rho(g)$ and $\lambda(g)$ are linearly ordered by
  $\subset$. Intuitively $\rho(g)$ is a set of ``minimal'' elements of
  $T(\Right(g))$ and $\l(g)$ is the set of ``maximal'' elements of
  $T(\Left(g))$. Now let $C\subseteq S^\k_\o$ be the set of all
  $\a$ satisfying
  \begin{enumerate}[label=(\roman*)]
  \item\label{it1} for all $f\in L(A)$, $\sup\ran(f)<\a\iff\sup\ran(h(f))<\a$,
  \item\label{it2} $A\cap \a$ is unbounded in $\a$,
  \item\label{it3} if $g\in T(L(B))$ and $\sup\ran(g)<\a$, then
    $\sup\{\sup\ran(f)\mid f\in\rho(g)\}<\a$ and
    $\sup\{\sup\ran(f)\mid f\in\lambda(g)\}<\a$,
  \item\label{it4} if $g\in T(L(B))$, $f\in T(\Left(g))$,
    $\sup\ran(g),\sup\ran(f)<\a$, and there exists $\hat f\in \Left(g)$
    such that $f\prec \hat f$ and $\hat f\not\subset f$, then there exists
    such an $\hat f$ with $\sup\ran(\hat f)< \a$,
  \item\label{it5} if $g\in T(L(B))$, $f\in T(\Right(g))$,
    $\sup\ran(g),\sup\ran(f)<\a$, and there exists $\hat f\in \Right(g)$
    such that $\hat f\prec f$ and $f\not\subset \hat f$, then there exists
    such an $\hat f$ with $\sup\ran(\hat f)< \a$,
  \end{enumerate}
 
  Our cardinality assumption on $\k$ guarantees that $C$ is a club. We
  will show that $C\cap A\subseteq B$ which is a contradiction. Let
  $\a\in C\cap A$ and let $f\in L(A)$ be such that
  $\sup\ran(f)=\a$. We will show that $\sup\ran(h(f))=\a$ and so
  $h(f)\in L(B)$ and $\a\in B$.  Suppose not. If $\sup\ran (h(f))<\a$,
  then by \ref{it1}, $\sup\ran(f)<\a$ which is a contradiction. So we
  can assume that $\sup\ran(h(f))>\a$. Because we assumed that
  $p(0)=0$ for all functions in question, there is $n_0<\o$ such that
  $h(f)(n_0)<\a\le h(f)(n_0+1)$. Let
  $$g=h(f)\rest (n_0+1).\eqno(*)$$ 
  In particular
  $$\sup\ran(g)<\a.\eqno(**)$$ 
  For every $m<\o$, pick $\a_m\in A$ such that $f(m)<\a_m<\a$. Such
  $\a_m$ exists by \ref{it2}. Now for each $m$ fix $f_m$ with
  $\sup\ran(f_m)=\a_m$ and $f_m\supset f\rest(m+1)$. We have two
  cases: either (A) $\sup\{m<\o\mid f_m\in \Left(g)\}=\o$ or (B)
  $\sup\{m<\o\mid f_m\in \Right(g)\}=\o$. We will show that both (A)
  and (B) lead to a contradiction with~\ref{it3}.

  \begin{claim} (A) If there are infinitely many $m<\o$ with
    $f_m\in \Left(g)$, then for all $m<\o$ we have
    $f\rest(m+1)\in\lambda(g)$ which contradicts~\ref{it3}.
  \end{claim}
  \begin{proof}
    For every $m$, there is $m'>m$ such that $f_{m'}\in \Left(g)$ and
    since $f\rest(m+1)\subseteq f\rest(m'+1)\subset f_{m'}$, we have
    that $f\rest(m+1)\in T(\Left(g))$. Suppose that
    $f\rest(m+1)\notin \lambda(g)$ for some $m$. Then by the
    definition of $\lambda(g)$, there exists $\hat f\in T(\Left(g))$
    such that $f\rest(m+1)\prec \hat f$, but
    $\hat f\not\subset f\rest(m+1)$.  We can w.l.o.g. assume that
    $\hat f\in \Left(g)$ and further, by \ref{it4}, that
    $\sup\ran(\hat f)<\a$.
    
    Since $f\rest(m+1)\prec\hat f$ and
    $\hat f\not\subset f\rest(m+1)$, there is $n\le m$ such that
    $\hat f(n)>f(n)$ and $n$ is smallest such that
    $\hat f(n)\ne f(n)$. This $n$ witnesses that $f\prec\hat f$.  So
    we have $h(f)\prec h(\hat f)$.  The latter implies that for the
    first $n'<\o$ with $h(f)(n')\ne h(\hat f)(n')$ we have
    $h(\hat f)(n')> h(f)(n')$. If $n'>n_0$ ($n_0$ is defined at $(*)$)
    then $\sup\ran(h(\hat f))\ge h(\hat f)(n')>h(f)(n')\ge \a$, a
    contradiction with \ref{it1}.  So $n'\le n_0$ and $h(\hat f)(n')>h(f)(n')=g(n')$,
    so we have $g\prec h(\hat f)$.  But this implies that
    $\hat f\in \Right(g)$ which is a contradiction again.  This proves
    the claim.
  \end{proof}
  
  \begin{claim} (B) If there are infinitely many $m<\o$ with
    $f_m\in \Right(g)$, then for all $m<\o$ we have
    $f\rest(m+1)\in\rho(g)$ which contradicts~\ref{it3}.
  \end{claim}
  \begin{proof}
    For every $m$, there is $m'>m$ such that $f_{m'}\in \Right(g)$ and
    since $f\rest(m+1)\subset f\rest(m'+1)\subset f_{m'}$, we have
    that $f\rest(m+1)\in T(\Right(g))$. Suppose that
    $f\rest(m+1)\notin \rho(g)$ for some $m$. Then by the definition
    of $\rho(g)$, there exists $\hat f\in T(\Right(g))$ such that
    $\hat f\prec f\rest(m+1)$, but $f\rest(m+1)\not\subset \hat f$.
    We can again assume that $\hat f\in \Right(g)$ and, by \ref{it5},
    that $\sup\ran\hat f<\a$.  There exists $n\le m$ with
    $\hat f(n)<f(n)$ and $n$ is the smallest such that
    $\hat f(n)\ne f(n)$.

    The number $n$ witnesses that $\hat f\prec f$ and so we must have
    $h(\hat f)\prec h(f)$.  The latter implies that for the first $n'<\o$
    with $h(f)(n')\ne h(\hat f)(n')$ we have $h(\hat f)(n')<h(f)(n')$. If
    $n'>n_0$, then $g\subset h(\hat f)$
    and hence $h(\hat f)\prec g$ which is a contradiction with
    $\hat f\in \Right(g)$.
    
    So $n'\le n_0$ and $h(\hat f)(n')<h(f)(n')=g(n')$, and again
    $h(\hat f)\prec g$, contradiction. This proves the claim.
  \end{proof}
  This completes the proof of Theorem~\ref{thm:ReducionNStoDLO}.
\end{proof}

\begin{theorem}[$V=L$]\label{thm:DLOVisLSiiCompl}
  If $\k>\o$ is a regular cardinal which is not the successor of an
  $\o$-cofinal cardinal, then $\qo_{\DLO}$ is $\Sii$-complete.
\end{theorem}
\begin{proof}
  By Theorem~\ref{thm:QOV=L} it is sufficient to reduce $\qo^\o$ to
  $\qo_{\DLO}$. But since $V=L$ every cardinal $\k>\o$ which is not
  the successor of an $\o$-cofinal cardinal satisfies the assumption
  of Theorem~\ref{thm:ReducionNStoDLO}.
\end{proof}

\begin{corollary}
  If $\k>\o$ is a regular cardinal which is not the successor of an
  $\o$-cofinal cardinal, then the embeddability of graphs $\qo_{G}$ is
  $\Sii$-complete. 
\end{corollary}
\begin{proof}
  It is a well known folklore that both embeddability and isomorphism
  of any model class can be coded into graphs (e.g. the authors of
  \cite{friedman1989borel} assume this without proof in the countable
  case).  We will sketch a proof for the sake of completeness in the case
  of linear orders (Theorem~\ref{thm:ToGraphs}).
\end{proof}

\begin{theorem}\label{thm:ToGraphs}
  For every $\k\ge \o$ there is a continuous function
  $F\colon \Mod^\k_{\DLO}\to \Mod^\k_G$ which preserves
  both embeddability and isomorphism.
\end{theorem}
\begin{proof}
  Given a linear order $\left(L,<_L\right)$ we will construct a graph
  $F\left(L,<_L\right)=(G,R)$ where $R=R\left(L,<_L\right)$ is a
  binary symmetric irreflexive relation on~$G=G\left(L,<_L\right)$.
  This construction will be such that it preserves both the embeddability
  and the isomorphism relations. Moreover it will be easy to see that
  if $F$ is translated through coding into a function from $\Mod^\k_{\DLO}$
  to $\Mod^\k_G$ it becomes continuous.

  The domain of the graph $G=G\left(L,<_L\right)$ consists of a copy of
  the domain of $L$ plus two vertices for every pair $a,b\in L$ such that
  $a<_L b$. Formally $G=L\cup (<_L\!\times \{0,1\})$.  The relation
  $R=R\left(L,<_L\right)$ is defined so that for every $a<_L b$
  the connections between $a,b,((a,b),0)$ and $((a,b),1)$ are as
  shown in~\eqref{eq:GLgraph}:
  \begin{equation}
    \xymatrix{((a,b),0) \ar@{-}[r]\ar@{-}[d]\ar@{-}[dr] & ((a,b),1) \ar@{-}[d] \\
      a & b} 
    \label{eq:GLgraph}
  \end{equation}
  Now any embedding $g\colon L_1\to L_2$ induces 
  an embedding 
  $$\hat g\colon G\left(L_1,<_{L_1}\right)\to G\left(L_2,<_{L_2}\right)$$
  by
  $$\hat g(a)=
  \begin{cases}
    g(a) & \text{if } a\in L_1,\\
    ((g(c_1),g(c_2)),0) & \text{if } a=((c_1,c_2),0)\in\ <_{L_1}\!\times\{0\},\\
    ((g(c_1),g(c_2)),1) & \text{if } a=((c_1,c_2),1)\in\ <_{L_1}\!\times\{1\}.
  \end{cases}
  $$
  This $\hat g$ is an isomorphism if and only if $g$ is.  On the other
  hand any embedding $g$ from $G\left(L_1,<_{L_1}\right)$ to
  $G\left(L_2,<_{L_2}\right)$ maps elements of $L_1$ to elements of
  $L_2$, because elements of $L_k$ are precisely the elements of
  $G\left(L_k,<_{L_k}\right)$ with an infinite
  $R\left(L_k,<_{L_k}\right)$-degree, $k\in\{1,2\}$. It is left to the reader
  to verify that the way the graph is defined ensures that
  $g\rest L_1$ is an embedding from $\left(L_1,<_{L_1}\right)$ to
  $\left(L_2,<_{L_2}\right)$.  Again this embedding is an isomorphism
  if and only if $g$ is.
\end{proof}

\subsection{Dichotomy for countable first-order theories in~$L$}

In \cite{HKM} it was proved that if $V=L$, $\kappa$ is a successor of
an uncountable regular cardinal~$\l$, then
$\cong_{T_1}\ \le_c\ \cong_{T_2}$ and
$\cong_{T_2}\ \not\le_B\ \cong_{T_1}$ holds for all $T_1$
classifiable and $T_2$ non-classifiable. This result can be improved
using Theorem~\ref{thm:maincor} together with some results
from~\cite{FHK}:

\begin{theorem}\label{thm:FHK86} (\cite[Thm~86]{FHK})
  Suppose that for all $\gamma<\kappa$, $\gamma^\omega<\kappa$ and $T$
  is a stable unsuperstable complete countable theory. Then
  $E^2_{\omega}\le_c\ \cong_T$. \qed
\end{theorem}

\begin{corollary}[$V=L$]\label{cor:stricstable}
  Suppose that $\kappa$ is regular and not the successor of an
  $\o$-cofinal cardinal and $T$ is a stable unsuperstable complete
  countable theory. Then $\cong_{T}$ is a $\Sigma^1_1$-complete
  relation.
\end{corollary}
\begin{proof}
  Follows from Theorems \ref{thm:FHK86} and~\ref{thm:maincor}
  and GCH in~$L$.
\end{proof}

\begin{theorem}(\cite[Thm~79]{FHK})
  \label{thm:FHK79}
  Suppose that $\kappa=\lambda^+=2^\lambda$ and
  $\lambda^{<\lambda}=\lambda$.
  \begin{enumerate}
  \item If $T$ is complete unstable or superstable with OTOP, then
    $E^2_{\lambda}\le_c\ \cong_T$.
  \item If $\lambda\ge 2^\omega$ and $T$ is complete superstable with DOP, then
    $E^2_{\lambda}\le_c\ \cong_T$. \qed
  \end{enumerate}
\end{theorem}

\begin{corollary}[$V=L$]\label{cor:compths}
  Suppose that $\kappa$ is the successor of a regular
  uncountable cardinal~$\l$. If $T$ is a non-classifiable complete
  countable theory, then $\cong_{T}$ is a $\Sigma^1_1$-complete
  relation.
\end{corollary}
\begin{proof}
  Follows from Theorems \ref{thm:maincor}, \ref{thm:FHK86},
  and~\ref{thm:FHK79}.
\end{proof}

By using yet another Theorem from~\cite{FHK} we obtain the following
dichotomy in~$L$. The class of $\Delta^1_1$ sets consists of sets $A$
such that both $A$ and the complement of $A$ are~$\Sii$~\cite{FHK}. 

\begin{theorem}[$V=L$]\label{thm:dichotomy}
  Suppose that $\kappa$ is the successor of a regular uncountable
  cardinal~$\l$. If $T$ is a countable first-order theory in a
  countable vocabulary, not necessarily complete, then one of the
  following holds:
  \begin{itemize}
  \item $\cong_T$ is $\Delta_1^1$.
  \item $\cong_T$ is $\Sigma_1^1$-complete.
  \end{itemize}
\end{theorem}
\begin{proof}
  For this proof it is useful to bare in mind how the isomorphism
  relation of a theory is defined, Definition~\ref{def:ISO}. Sometimes
  in literature it is defined differently, but these are mutually 
  Borel-bi-reducible (there is a Borel reduction both ways).
  
  It has been shown \cite[Thm~70]{FHK} that if a complete theory $T$
  is classifiable, then $\cong_T$ is $\Delta_1^1$.  So for a
  complete countable theory $T$ the result follows from
  Corollary~\ref{cor:compths}. Suppose $T$ is not a complete theory.  Let
  $\L$ be the vocabulary of $T$ and $\{T_\a\}_{\a<2^\o}$ be
  the set of all the complete theories in $\L$ that extend
  $T$. Notice that $\cong_T=\bigcap_{\a<2^{\o}}\cong_{T_\a}$,
  therefore if $\cong_{T_\a}$ is a $\Delta_1^1$ equivalence
  relation for all $\a<\k$, then so is~$\cong_T$ since $2^\o<\k$.
  
  Suppose $T'$ is a complete countable theory in $\L$ that
  extends $T$ such that $\cong_{T'}$ is not a $\Delta_1^1$
  equivalence relation. Then $T'$ is a non-classifiable countable
  theory. By Corollary \ref{cor:compths} $\cong_{T'}$ is a
  $\Sigma_1^1$-complete equivalence relation. We will show that
  $\cong_{T'}\ \le_B\ \cong_T$ which finishes the proof. Define
  $\F\colon\k^\k\to\k^\k$ by $$\F(\eta)=
  \begin{cases}
    \eta &\mbox{if } \A_\eta\models T'\\
    \xi & \mbox{otherwise. }
  \end{cases}
  $$
  where $\xi$ is a fixed element of $\k^\k$ such that
  $\A_\xi\not\models T'$. Since $T'$ extends $T$,
  $\eta\ \cong_{T'}\ \zeta \Leftrightarrow \F(\eta)\
  \cong_T\ \F(\zeta)$.
  To show that $\F$ is Borel, note
  that $$\F^{-1}([\eta\restriction\a])=
  \begin{cases}
    [\eta\restriction\a]\backslash
    \{\zeta\mid\A_\zeta\not\models T'\} &\text{if } \xi\not
    \in[\eta\restriction\a]\\
    \{\zeta\mid\A_\zeta\not\models T'\} \cup
    [\eta\restriction\a] &\text{if } \xi\in [\eta \restriction\a].
  \end{cases}
  $$
  Since $[\eta\restriction\a]$ is a basic open set and
  $\{\zeta\mid\A_\zeta\not\models T'\}$ is a Borel set,
  $[\eta\restriction\a]\backslash
  \{\zeta\mid\A_\zeta\not\models T'\}$
  and
  $[\eta\restriction\a]\cup \{\zeta\mid\A_\zeta\not\models
  T'\}$ are Borel sets.
\end{proof}

The dichotomy of Theorem~\ref{thm:dichotomy} is not provable in~ZFC.
In \cite{HSa,HSb} it was shown, assuming $\k$ is a successor and
$\k\in I[\k]$, that there is a stable unsuperstable countable theory
$T$ in a countable vocabulary such that $\cong_T$ is $\Borel^*$ (a
generalization of Borel sets to non-well-founded
trees~\cite{FHK,HalShe}). By Theorem~\ref{thm:FHK86} $\cong_T$ is not
$\Delta^1_1$, if $E^2_\o$ is not. It was proved in \cite{HK18} that
there is a model of ZFC where $\Borel^*\subsetneq \Sii$ (unlike in
$L$,~\cite{HK}), in which $E^2_\o$ is not~$\Dii$, and in which
$\k\in I[\k]$ for successor~$\k$.


\section{The case $V\ne L$}

\subsection{$\Sii$-completeness of $\qo^{\NS}$ for weakly ineffable~$\k$}
\label{ssec:Ineffable}

In Section \ref{sec:Corollaries} we answered the questions
\cite[Q.~3.47]{KLLS}, \cite[Q.'s~11.3~and~11.4]{Luc} and
\cite[Q.~15]{FHK15} under $V=L$. We used Theorem~\ref{thm:maincor} as
the starting point. But what if $V\ne~L$? In this section we provide
further partial answers to \cite[Q.'s~11.3~and~11.4]{Luc} outside
of~$L$.  Recall that these questions ask ``Given a weakly compact
cardinal $\k$, are $\qo^{\NS}$ and $\qo_{\DLO}$ complete for $\Sii$
quasi-orders?''  Recall that $\qo_G$ is the embeddability of graphs,
Definition~\ref{def:ISO}. We will use the following theorem:

\begin{theorem}(\cite[Cor~10.24]{Luc}) \label{thm:biembedcomplete} If
  $\k$ is weakly compact, then both the quasi-order of embeddability
  and the equivalence relation of bi-embeddability of graphs,
  $\qo_{G}$ and $\approx_{G}$ respectively, are
  $\Sigma_1^1$-complete.
\end{theorem}

\begin{definition}[Weakly compact diamond]\label{WCDiamond}
  Let $\k>\o$ be a cardinal. The \emph{weakly compact ideal} is
  generated by the sets of the form
  $\{\a<\k\mid \la V_\a,\in,U\cap V_\a\ra \models \lnot\f\}$ where
  $U\subseteq V_\k$ and $\f$ is a $\Pii$-sentence such that
  $\la V_\k,\in,U\ra \models \f$.  A set $A\subseteq\k$ is said to be
  \emph{weakly compact}, if it does not belong to the weakly compact
  ideal.  Note that $\k$ is weakly compact if and only if there exists
  $A\subseteq\k$ which is weakly compact, i.e. the weakly compact ideal
  is proper.  For weakly compact $S\subseteq \k$, the $S$-\emph{weakly
    compact diamond}, $\WC_\k(S)$, is the statement that there exists
  a sequence $(A_\a)_{\a<\k}$ such that for every $A\subseteq S$ the set
  $$\{\a<\k\mid A\cap\a=A_\a\}$$
  is weakly compact. We denote $\WC_\k=\WC_\k(\k)$.
\end{definition}

Weakly compact diamond was originally introduced in \cite{Sun} and
thoroughly analyzed in \cite{Hell}. In \cite{AHKM} it was used to
study the reducibility properties of $E^\k_{\reg}$. It has been sometimes
called the \emph{dual diamond}.

\begin{fact}\label{fact:WCD}
  If $\k$ is weakly ineffable (same as almost ineffable), then
  $\WC_\k$ holds.  See \cite{Hell} for proofs and references.
\end{fact}

The proof of Lemma~\ref{lemma:WC_dual} can be found in \cite{AHKM} in
complete detail.

\begin{lemma}\label{lemma:WC_dual}
  Let $\k$ be a weakly compact cardinal.  The weakly compact diamond
  $\WC_\k$ implies the following principle $\WC^*_\k$.  There exists a
  sequence $\la f_\a \ra_{\a\in\reg(\k)}$ such that
  \begin{itemize}
  \item $f_\a\colon\a\to \a$,
  \item for all $g\in\k^\k$ and stationary $Z\subseteq \k$ the set
    $$\{\a\in \reg(\k)\mid g\rest\a=f_\a\land \a\cap Z\text{ is stationary}\}$$
    is stationary. \qed
  \end{itemize}
\end{lemma}

Let us introduce a version of $\WC^*_\k$ for graphs, denoted~$\WC^*_G$.  
Let $G_{<\k}=\{\G_\beta\}_{\beta<\k}$ be an enumeration of
all graphs with domain some ordinal $\a<\k$. For all $\a<\k$, let
$G_{<\a}=\{\G_\beta\}_{\beta<\a}$. 

The principle $\WC^*_G$ states that there exists a sequence
$\la f_\a\ra_{\a<\k}$ such that
\begin{itemize}
\item $f_\a\in (G_{<\a})^\a$,
\item if $(S,g)$ is a pair such that $S\subseteq \k$ is stationary and
  $g\in (G_{<\k})^\k$, the set
  $$\{\a\in \reg(\k)\mid g\restriction \a=f_\a\wedge S\cap\a\text{ is
    stationary}\}$$ is stationary.
\end{itemize}

\begin{fact}\label{fact:mahlodd}
  If $\WC^*_\k$ holds, then $\WC^*_G$ holds.
\end{fact}
\begin{proof}
  Let $\la\bar{f}_\a\ra_{\a<\k}$ be a sequence that witnesses
  $\WC^*_\k$. Define the
  sequence $\la f_\a\ra_{\a<\k}$ by
  $f_\a(\beta)=\G_{\bar{f}_\a(\beta)}$.
  
  To show that $\la f_\a\ra_{\a<\k}$ witnesses $\WC^*_G$, let
  $g\in (G_{<\k})^\k$ be any function and $S\subseteq \k$ a 
  stationary set. There is a function $\bar{g}\colon\k\to \k$ such
  that $g(\a)=\G_{\bar{g}(\a)}$. Because of $\WC^*_\k$ we
  know that the set
  $$\{\alpha\in \reg(\k)\mid \bar{g}\restriction
  \alpha=\bar{f}_\alpha\wedge Z\cap\alpha\mbox{ is stationary}\}$$
  is stationary. From the definitions of $\la f_\a\ra_{\a<\k}$ and $\bar{g}$ 
  it follows that the set
  $$\{\alpha\in \reg(\k)\mid g\restriction \alpha=f_\alpha\wedge
  Z\cap\alpha\mbox{ is stationary}\}$$ is stationary.
\end{proof}

\begin{definition}\label{def:closure_points}
  A \emph{closure point} of a function $s\colon\k\to\k$ is an ordinal
  $\a<\k$ such that for all $\b<\a$ we have $s(\b)<\a$. The set of all
  closure points of $s$ is a club.
\end{definition}

\begin{theorem}\label{thm:reductionreg}
  If $\k$ is weakly compact and $\WC^*_G$ holds, then $\qo^{\reg}$ as
  well as $\qo^{\NS}$ are $\Sigma_1^1$-complete.
\end{theorem}
\begin{proof}
  The claim for $\qo^{\NS}$ follows from Fact~\ref{reg-red-NS} once we
  prove the claim for $\qo^{\reg}$.
  By Theorem \ref{thm:biembedcomplete} it is enough to show that
  $\qo_{G}\ \le_B\ \qo^{\reg}$.  For all $K,H\in G_{<\k}$ we
  write $K\qo H$ if $K$ is embeddable to $H$.  Let us denote
  by $Q$ the quasi-order $((G_{<\k})^\k, \le_Q)$, where $f\le_Q g$
  holds if there is a club $C$ such that for all $\a\in C$,
  $f(\a)\qo g(\a)$ holds.
  
  Note that every $\G_\a$ equals some $\A_{p}$, $p\in \k^{<\k}$, where
  $\A_p$ is as defined in Definition~\ref{def:clubstr}.  Vice versa,
  if $\eta\in \Mod^\k_G$ (i.e. is a code for a graph, that is,
  $\A_\eta$ is a graph, Definition~\ref{def:coding}), then for every
  $\a<\k$ there is $\b<\k$ such that $\A_{\eta\rest\a}=\G_\b$.  Let
  $H$ be the graph with domain $2$ and no edges. Define
  $F\colon \Mod^\k_G\to (G_{<\k})^\k$ by $$F(\eta)(\a)=
  \begin{cases}
    \A_{\eta\restriction\a} &\mbox{if } \a\in C_\pi\\
    H & \mbox{otherwise. }
  \end{cases}
  $$
  where $C_\pi$ is as in Definition~\ref{def:clubstr}.
   \begin{claim}
    $\eta \qo_{G} \xi$ if and only if
    $F(\eta)\le_Q F(\xi)$.
  \end{claim}
  
  \begin{proof}
    Let us show that if $\eta \qo_{G} \xi$, then
    $F(\eta)\le_Q F(\xi)$. Suppose $\eta\qo_{G}\xi$, then
    there is an embedding $f\colon\k\to\k$  of $\A_\eta$ to $\A_\xi$.
    Let $D$ be the set of all closure points of~$f$
    (Definition~\ref{def:closure_points}). Since $D$ is a club,
    $f\restriction\a$ is an embedding of $\A_{\eta\restriction\a}$ to
    $\A_{\xi\restriction\a}$, for all $\a\in D\cap C_\pi$. We conclude
    that $F(\eta)\le_Q F(\xi)$.  Let us show that if
    $\eta \not\qo_{G}\xi$, then
    $F(\eta)\not\le_Q F(\xi)$. Suppose
    $\eta\not\qo_{G}\xi$. The property
    \begin{center}
      ``(there is no embedding of $\A_\eta$ to $\A_\xi$) $\land$ ($\k$ is
      regular) $\land$ ($C_\pi$ is unbounded)''
    \end{center}
    \noindent is a $\Pi_1^1$-property of the structure $(V_\k,\in,A)$, where
    $A=(\eta\times\{0\})\cup(\xi\times\{1\})\cup (C_\pi\times \{2\})$.
    Since $\k$ is weakly compact, there is stationary many
    ordinals $\gamma$ such that the above property holds with $\eta$ and
    $\xi$ replaced by $\eta\rest\gamma$ and $\xi\rest\gamma$ as well as $\k$
    replaced by $\g$, i.e.
    \begin{center}
      ``(there is no embedding of $\A_{\eta\rest\g}$ to
      $\A_{\xi\rest\gamma}$) $\land$ ($\g$ is regular) $\land$
      ($C_\pi\cap \g$ is unbounded in $\g$)''.
    \end{center}
    We conclude that there are stationary many ordinals $\gamma$ such
    that $F(\eta)(\gamma)\not\qo F(\xi)(\gamma)$, hence
    $F(\eta)\not\le_Q F(\xi)$.
  \end{proof}
  Let $\la f_\a\ra_{\a<\k}$ be a sequence that witnesses $\WC^*_G$.
  For all $\a\in \reg(\k)$ define the relation $\le^\a_Q$ on
  $(G_{<\k})^\a$ by: $f\le^\a_Q g$ if there is a club $C\subseteq \a$
  such that for all $\beta\in C$, $f(\beta)\qo g(\beta)$
  holds. Notice that since the intersection of two clubs is a club,
  then $\le^\a_Q$ is a quasi-order.  Define the map
  $\F\colon (G_{<\k})^\k\to 2^\k$ by $$\F(f)(\a)=
  \begin{cases}
    0 &\mbox{if } f\restriction\a\le^\a_Q f_\a\\
    1 & \mbox{otherwise }
  \end{cases}
  $$
  \begin{claim}
    $f\le_Q g$ if and only if
    $\F(f)\qo^{\reg}\F(g)$.
  \end{claim}
  
  \begin{proof}
    Let us show that if $f\le_Q g$, then
    $\F(f)\qo^{\reg}\F(g)$. Suppose $f\le_Q g$,
    then there is a club $C\subseteq\k$ such that for all $\a\in C$,
    $f(\a)\qo g(\a)$. Therefore, for all
    $\a\in \lim(C)\cap \reg(\k)$ it holds that
    $f\restriction\a\le^\a_Q g\restriction\a$. Now if
    $\a\in \lim(C)\cap \reg(\k)$ is such that $\F(g)(\a)=0$, then
    $g\restriction\a\le^\a_Q f_\a$, so $f\restriction\a\le^\a_Q f_\a$
    and $\F(f)(\a)=0$. We conclude that
    $(\F(f)^{-1}[1]\backslash \F(g)^{-1}[1])\cap
    \reg(\k)$
    is non-stationary. Hence
    $\F(f)\qo^{\reg}\F(g)$.  Let us show that if
    $f\not\le_Q g$, then
    $\F(f)\not\qo^{\reg}\F(g)$. Suppose that
    $f\not\le_Q g$, then there is a stationary set $S\subseteq \k$
    such that for all $\a\in S$, $f(\a)\not\qo g(\a)$. Because
    of $\WC^*_G$ we know that the set
    $$A=\{\a\in \reg(\k)\mid g\restriction\a=f_\a\land S\cap \a\text{
      is stationary}\}$$
    is a stationary set. Therefore, for all $\a\in A$,
    $\F(g)(\a)=0$, and for all $\beta\in S\cap\a$,
    $f(\beta)\not\qo g(\beta)$. Since for all $\a\in A$,
    $g\restriction\a=f_\a$, and $S\cap\a$ is stationary, we conclude
    that $f\restriction\a\not\le_Q^\a f_\a$ holds for all $\a\in A$.
    Hence, for all $\a\in A$, $\F(g)(\a)=0$ and
    $\F(f)(\a)=1$. We conclude that
    $A\subseteq (\F(f)^{-1}[1]\backslash
    \F(g)^{-1}[1])\cap\ \reg(\k)$,
    and since $A$ is stationary,
    $\F(f)\not\qo^{\reg}\F(g)$.
  \end{proof}
  Clearly $\F\circ F\colon\Mod^\k_G\to 2^\k$ is a
  Borel-reduction of $\qo_{G}$ to $\qo^{\reg}$.
\end{proof}

\begin{theorem}\label{main}
  If $\k$ is weakly ineffable, then $\qo^{\NS}$ is
  $\Sigma_1^1$-complete.
\end{theorem}
\begin{proof}
  Follows from Fact \ref{fact:WCD}, Lemma \ref{lemma:WC_dual},
  Fact~\ref{fact:mahlodd}, and Theorem~\ref{thm:reductionreg}.
\end{proof}

Thus, the only case concerning \cite[Q.~11.4]{Luc} that is still open
is the case where $V\ne L$ and $\k$ is a weakly compact, but not 
weakly ineffable cardinal. For example the first weakly compact
is such~\cite[Lemma~1.12]{friedman2001subtle}. For successor cardinals, we know from \cite{FWZ} that it can be forced the relation $E_{NS}^2$ to be a $\Delta_1^1$ equivalence relation. So it is consistently true that $\qo^{\NS}$ is not $\Sigma_1^1$-complete.

\subsection{$\Sii$-completeness of $\cong_{\DLO}$ and $\cong_{G}$ for
  weakly compact~$\k$}
\label{ssec:SiiCompIsoDLO}

In this section we prove:

\begin{theorem}\label{thm:IsoisSii}
  Suppose that $\k$ is weakly compact. Then the isomorphism relation
  on dense linear orders is $\Sii$-complete.
\end{theorem}

Before proving
Theorem~\ref{thm:IsoisSii}, we first prove the following:

\begin{lemma}\label{lem:bi-emb_to_Ereg}
  If $\k$ is weakly compact, then the bi-embeddability of graphs
  $\approx_G$ is reducible to $E^\k_{\reg}$
  (Definition~\ref{def:ISO}).
\end{lemma}
\begin{proof}
  Let $C_\pi$ be the club as in Definition \ref{def:clubstr} and for all
  $\a\in C_\pi$ define the relation $\approx_{G}^\a$ as follows. For
  all $\eta,\xi\in \Mod^\k_G$, let $\eta\ \approx_{G}^\a\ \xi$, if
  $\A_{\eta\rest\a}$ is embeddable in
  $\A_{\xi\rest\a}$ and $\A_{\eta\rest\a}$ is
  embeddable in $\A_{\xi\rest\a}$ (Definition~\ref{def:clubstr}).
  
  There are at most $\k$ many equivalence classes of $\approx_G^\a$,
  so let $g_\a\colon \Mod^\k_G\to\k$ be a function with the property
  that for all $\eta,\xi\in \Mod^\k_G$ we have $g_\a(\eta)=g_\a(\xi)$
  if and only if \mbox{$\eta\approx_G^\a \xi$}.
  
  Define the reduction $\F\colon \Mod^\k_G\to \k^\k$
  by
  $$\F(\eta)(\a)=
  \begin{cases}
    g_\a(\eta) &\mbox{if } \a\in C_\pi\\
    0 & \mbox{otherwise.}
  \end{cases}
  $$ 
  Let us show that if $\eta\approx_{G}\xi$, then
  $(\F(\eta),\F(\xi))\in E^\k_{\reg}$. Suppose that $\eta\approx_{G}\xi$.
  Then there are embeddings $F_1\colon \k\to\k$ and $F_2\colon\k\to\k$
  from $\A_\eta$ to $\A_\xi$, and from $\A_\xi$ to $\A_\eta$
  respectively. Let $D_1$ and $D_2$ be the sets of all closure points
  (Definition~\ref{def:closure_points}) of $F_1$ and $F_2$
  respectively. These are closed unbounded sets in~$\k$.  Then for all
  $\a\in D_1\cap D_2\cap C_\pi$, $\A_{\eta\restriction\a}$ and
  $\A_{\xi\restriction\a}$ are bi-embeddable. Hence for all
  $\a\in D_1\cap D_2\cap C_\pi$, $\F(\eta)(\a)=\F(\xi)(\a)$. We
  conclude that $(\F(\eta),\F(\xi))\in E^\k_{\reg}$.
  
  Let us show that if $\eta\not\approx_{G}\xi$, then $\F(\eta)$ and
  $\F(\xi)$ are not $E^\k_{\reg}$-equivalent. Suppose that
  $(\eta,\xi)\notin\ \approx_{G}$, without loss of generality, suppose
  that there is no embedding of $\A_\eta$ into
  $\A_\xi$. The property
  
  \textit{There is no embedding of $\A_\eta$ to
    $\A_\xi\ \wedge\ \k$ is regular $ \wedge\ C_\pi$ is
    unbounded}
  
  is a $\Pi_1^1$-property of the structure $(V_\k,\in,A)$, where
  $A=(\eta\times\{0\})\cup(\xi\times\{1\})\cup ( C_\pi\times \{2\})$.
  Since $\k$ is weakly compact, there are stationary many ordinals
  $\gamma<\k$ such that $C_\pi\cap\gamma$ is unbounded,
  $\gamma\in C_\pi$, $\gamma$ is regular, and \textit{there is no
    embedding of $\A_{\eta\restriction\gamma}$ to
    $\A_{\xi\restriction\gamma}$}. We conclude that there are
  stationary many regular cardinals $\gamma$ with
  $\F(\eta)(\gamma)\neq\F(\xi)(\gamma)$, hence
  $\eta$ and $\xi$ are not $E^\k_{\reg}$-equivalent.
\end{proof}

\begin{corollary}\label{cor:WCEkappareg}
  If $\k$ is weakly compact, then $E^\k_{\reg}$ is $\Sii$-complete.
\end{corollary}
\begin{proof}
  Follows from Theorem \ref{thm:biembedcomplete} and
  Lemma~\ref{lem:bi-emb_to_Ereg}.
\end{proof}

Now we can prove Theorem~\ref{thm:IsoisSii}:

\begin{proof}[Proof of Theorem~\ref{thm:IsoisSii}]
  By \cite[Thm 3.9]{AHKM} we have $E^\k_{\reg}\ \le_c\ \cong_{\DLO}$,
  so the result follows from Corollary~\ref{cor:WCEkappareg}
\end{proof}

By Theorem~\ref{thm:ToGraphs} we get the following corollary to
Theorem~\ref{thm:IsoisSii}:

\begin{corollary}
  Suppose that $\k$ is weakly compact. Then the isomorphism relation
  on graphs is $\Sii$-complete. \qed
\end{corollary}


\providecommand{\bysame}{\leavevmode\hbox to3em{\hrulefill}\thinspace}
\providecommand{\MR}{\relax\ifhmode\unskip\space\fi MR }
\providecommand{\MRhref}[2]{%
  \href{http://www.ams.org/mathscinet-getitem?mr=#1}{#2}
}
\providecommand{\href}[2]{#2}

\end{document}